\documentclass{amsart}

\usepackage{amsfonts}
\usepackage{amsmath,bm}
\usepackage{amssymb}
\RequirePackage{mathtools, amsthm, graphicx, mathrsfs, dsfont}
\allowdisplaybreaks

\newcommand{\A}{\mathcal{A}}
\newcommand{\C}{\mathbb{C}}
\newcommand{\R}{\mathbb{R}}
\newcommand{\N}{\mathbb{N}}
\newcommand{\Z}{\mathbb{Z}}
\let\epsilon\varepsilon
\newcommand{\om}{\boldsymbol{\omega}}
\def\bbe{{\bf e}}

\def\bX{\boldsymbol{X}}
\def\bx{\boldsymbol{x}}
\def\bY{\boldsymbol{Y}}
\def\bs{\boldsymbol{s}}
\def\bzero{\boldsymbol{0}}
\def\bone{\boldsymbol{1}}
\def\bv{\boldsymbol{v}}

\def\bK{\boldsymbol{K}}

\providecommand{\norm}[1]{\lvert\lvert #1 \rvert\rvert}

\newtheorem{theorem}{Theorem}[section]
\newtheorem{lemma}[theorem]{Lemma}

\newtheorem{prop}[theorem]{Proposition}
\newtheorem{corollary}[theorem]{Corollary}

\theoremstyle{definition}
\newtheorem{definition}[theorem]{Definition}

\theoremstyle{remark}
\newtheorem{remark}[theorem]{Remark}

\numberwithin{equation}{section}

\newcommand{\abs}[1]{\lvert#1\rvert}
\newcommand{\floor}[1]{\lfloor#1\rfloor}

\RequirePackage{mathtools, amsthm, graphicx, mathrsfs, dsfont}

\makeatletter
\@namedef{subjclassname@2020}{\textup{2020} Mathematics Subject Classification}
\makeatother

\begin{document}

\title{Quantitative weak mixing of self-affine tilings}

\author{Juan Marshall-Maldonado}
\address{Department of Mathematics, Bar-Ilan University, Ramat Gan.}
\email{jgmarshall21@gmail.com}




\begin{abstract}
We study the regularity of spectral measures of dynamical systems arising from a translation action on tilings of substitutive nature. The results are inspired in the work of Bufetov and Solomyak \cite{bufetov2014modulus}, where they established a log-Hölder modulus of continuity of one-dimensional self-similar tiling systems. We generalize this result to higher dimensions in the more general setting of self-affine tilings systems. Further analysis leads to uniform estimates in the whole space of spectral parameters, allowing to deduce logarithmic rates of weak mixing.
\end{abstract}

\date{\today}


\maketitle



\section{Introduction}
Crystallography is the study of the atomic structure of crystals, typically through the widely used X-ray diffraction technique. For many years, only materials that produced diffraction patterns with ``ordered'' structures were known. In 1982, materials scientist Dan Schechtman discovered certain aluminium-manganese alloys whose diffraction patterns exhibited rotational symmetries of order 5—an arrangement that cannot occur in patterns with translational symmetries, as dictated by the crystallographic restriction theorem. This groundbreaking discovery earned Schechtman the Nobel Prize in Chemistry in 2011.

The discovery of these materials was partially anticipated by prior literature describing tilings and \textit{Delone sets} without translational symmetries. A tiling is a countable collection of tiles that fully covers a space without any overlaps. British physicist Roger Penrose published examples of tilings that exhibited diffraction patterns without translational symmetries.

The above-mentioned diffraction problems are closely related to the spectral study of dynamical systems of tilings. In dynamical terms, these diffraction patterns, often called Bragg peaks in quasicrystal literature, correspond to eigenvalues of the Koopman representation linked to the translation action on the tiling space (see \cite{baake2015dynamical} for details). 

A very rich structure of the tiling system appears when we consider tilings that come from some \textit{inflation rule} or \textit{substitution}: a substitution $\zeta$ maps each tile in a finite set to some “supertile” using a fixed expanding linear transformation, then decomposes it into a finite union of non-overlapping copies of the original tiles, similar to assembling a jigsaw puzzle. Fixed points of the substitution (when acting on tilings) are called \textit{self-affine} tilings. If the expanding map is a similarity, the tiling is called \textit{self-similar}. Substitutive tilings are a rich source of aperiodic structures, with applications extending beyond mathematics to physics and material science. The interested reader may consult the surveys \cite{baake2016aperiodic,articletrevino} or the book \cite{baake2013aperiodic} for a deeper introduction.

In the context of dynamical systems, the substitution acts as a renormalization tool that aids in understanding long orbits of the natural translation action of $\R^d$ on tiling spaces. More precisely, for a tiling $\bX$ of $\R^d$ and a vector ${\bs}\in \R^d$, define $\phi_{{\bs}}(\bX)$ as the tiling obtained from translating each tile in $\bX$ by $-{\bs}$. We examine the spectrum of the Koopman representation of the translation flow for certain dynamical systems derived from self-similar and self-affine tilings of $\R^d$. We focus on the decay of spectral measures for weakly-mixing systems, excluding cases where the system has non-trivial eigenvalues.

This work is based on the unpublished master thesis of the author \cite{marshall2017modulo} (in Spanish): we correct some arguments and extend the results of this thesis. The aim is to extend a result of A. Bufetov and B. Solomyak in \cite{bufetov2014modulus}, from one-dimensional substitution tilings to higher dimensions. For a comprehensive overview of the spectrum of substitutions, see \cite{bufetov2023self}.

The following is a list of some recent related work relevant to us:
\begin{itemize}
	\item In \cite{bufetov2013limit} A. Bufetov and B. Solomyak obtained upper bounds for deviations of ergodic averages for self-similar tilings.
	\item In \cite{schmieding2018self} S. Schmieding and R. Treviño showed upper bounds for deviations of ergodic averages for self-affine tilings.
	\item In \cite{Jemme} J. Emme generalize Theorem 6.2 from \cite{bufetov2014modulus} for a self-similar substitution tiling: he obtains a modulus of continuity around zero for self similar tilings.
	\item In \cite{bufetov2018holder,JEP_2021__8__279_0} A. Bufetov and B. Solomyak and G. Forni in \cite{forni2022twisted} proved quantitative weak mixing of generic translations surfaces of every genus greater or equal than 2. This is achieved generalizing the setup developed in \cite{bufetov2014modulus} for substitutions, to the S-adic setting in the former case, and in the latter case following a cohomological approach.
	\item In \cite{avila2023quantitative} A. Avila, G. Forni and P. Safaee obtained quantitative weak mixing for interval exchange transformations.
	\item In \cite{trevino2020quantitative} R. Treviño showed a generalization Theorem 4.1 of \cite{bufetov2014modulus} in the context of random deformations of multiple substitution tilings of higher dimensions.
	\item In \cite{solomyak2024spectral} B. Solomyak and R. Treviño built a spectral cocycle for higher dimensional tilings (generalizing the one of  \cite{bufetov2020spectral}) to obtain lower dimension estimates on the spectral measures of pseudo self-similar tilings.
	\item Related work on the absence of absolutely continuous component of the spectral measures is done in \cite{baake2019geometric,baake2019renormalisation}.
\end{itemize}

The aim of the present paper is to study the decay of the spectral measures of substitution tiling systems of $\R^d$ and the decay of Cesàro means of correlations. Since the self-affine tiling case requires additional assumptions compared to the self-similar case, we present the results separately. In both cases, the main assumption is that the eigenvalues of the substitution expansion have Galois conjugates with moduli greater than 1 that lie outside the linear expansion's spectrum. In this case, we will say the substitution is \textit{strongly totally non-Pisot}. The formal definition is given in Definition \ref{def:stnp}.

In Theorem \ref{thm:main} and \ref{thm:mainselfaffine} we prove uniform estimates for spectral measures, in the self-similar and self-affine case respectively. As a consequence of this regularity, we deduce explicit decay of Cesàro means of correlations (see for example \cite{Knill}). This is also the strategy followed in \cite{bufetov2018ergodic,avila2023quantitative} to prove quantitative weak mixing. In simplest terms, for two observables $f,g\in L^2$ of the tiling space $\mathfrak{X}_\zeta$ associated with the substitution, we show an inequality of the form 
\[
\dfrac{1}{R^d} \int_{[-R,R]^d} \left| \langle f \circ \phi_{{\bs}}, g \rangle_{L^2} \right|^2 d{\bs} \leq C_{f,g}\log(R)^{-\alpha},
\]
for every $R\geq2$, for some constants $\alpha,C_{f,g}>0$. This is the content of Theorems \ref{thm:corr} and \ref{thm:corr2} (self-similar/self-affine case).

The rest of the paper is organized as follows: Section 2 provides necessary background of tiling systems and their spectral theory. In Section 3 we show some preliminary results on the structure of twisted Birkhoff integrals and recall some conditions for weak mixing of self-affine tilings. In Section 4 we prove the key result for the regularity of spectral measures. In Section 5 we prove quantitative weak mixing for self-similar tilings (Theorem \ref{thm:corr}) by proving the regularity of spectral measures in Theorem \ref{thm:main}. Finally, Section 6 addresses the more subtle case of self-affine tilings.

\subsection{Acknowledgments}I would like to express my gratitude to Alejandro Maass who supervised my master's thesis at Universidad de Chile, which served as a basis for this work. I would also like to thank Boris Solomyak and Rodrigo Treviño for reading early versions of this manuscript and suggesting many improvements. My research is supported by the Israel Science Foundation grant \#1647/23.

\section{Background on tiling spaces}
\subsection{Tiling dynamical systems}
Most of the definitions and results stated in this section may be found in \cite{robinson2004symbolic} or \cite{solomyak1997dynamics}. Let a \textbf{tile} (of $\R^d$) to be a pair $T=(\text{supp}(T),\text{punc}(T))$, where $\text{supp}(T)$, the \textbf{support} of $T$, is a compact subset of $\R^d$ which is the closure of its interior and $\text{punc}(T)$, the \textbf{puncture} of $T$ is a point in the interior of $\text{supp}(T)$. A \textbf{patch} is a finite set of tiles $\boldsymbol{P} = \{P_1,\dots, P_n\}$ for which $\text{int}(\text{supp}(P_i)) \cap \text{int}(\text{supp}(P_j)) = \emptyset$, for $i\neq j$ ($\text{int}(A)$ denotes the interior of a set $A\subseteq\R^d$). The support of a patch is the union of the supports of the tiles conforming it. A \textbf{tiling} is an infinite collection of tiles $\bX = \{X_i\}_{i\in\N}$ such that every finite sub-collection defines a patch and covers the whole space, i.e., $\R^d = \bigcup_{i\in\N} \text{supp}(X_i)$.

Let ${\bs} \in \R^d$ and $T$ a tile. The \textbf{translate} of $T$ by ${\bs}$ is the tile $T+{\bs} = (\text{supp}(T)+{\bs}, \text{punc}(T)+{\bs})$ and this notion defines an equivalence relation: the class of a tile is the set of all translations of it. The notion of translate is extended naturally to patches and tilings also. We will be interested only in tilings conformed of translations of a fixed finite set of nonequivalent tiles $\A = \{T_1,\dots,T_m\}$ that we call \textbf{alphabet}, and their elements \textbf{prototiles}. Because these tiles represent all their translations, we assume $\text{punc}(T_j) = 0$, for all $j=1,\dots,m$.

Denote by $\mathcal{P}_{\A}$ the patches conformed only by translates of tiles in $\A$, and the \textbf{full-tiling space} associated to an alphabet $\A$ by 
\[
\mathfrak{T}_{\A} = \left\{ \bX = \left\{X_i\right\}_{i\in\N} \text{ tiling } | \: \forall i\geq1, X_i \in \mathcal{P}_\A \right\}.
\]
When $d>1$, it is not a trivial task to decide whether this set is empty or not, but we will always assume that given an alphabet there are tilings conformed only by translations of tiles in the alphabet.

A well-known metric for $\mathfrak{T}_{\A}$ is the following: for a tiling $\bX$ and a subset $F\subseteq \R^d$ denote by $]F[^{\bX}$ the patch given by the tiles of $\bX$ whose supports are contained in $F$. Then, a metric Dist on $\mathfrak{T}_{\A}$ is defined by
\small{
\begin{align*}
\text{Dist}(\bX,\bY) &= \min \left( \: \text{D}(\bX,\bY),\: 1/\sqrt{2} \: \right), \\
\text{D}(\bX,\bY) &= \inf \: \left\{ \epsilon > 0 \: | \: \exists\: {\bs} \in B_\epsilon: \: \left]B_{1/\epsilon}\right[^{\bX} = \left]B_{1/\epsilon}\right[^{\bY} + {\bs} \: \right\},
\end{align*}}where $B_r$, denotes the Euclidean ball centered at the origin with radius $r>0$. In plain language, this metric consider close two tilings that coincide in a large central patch after a small translation. It is desirable to have a compact space, but the metric introduced does not ensure $\mathfrak{T}_{\A}$ have this property. To do so, we will always assume that every tiling in the full-tiling space has \textbf{finite local complexity} (f.l.c.), this means that for every tiling in the space, the number of different patches conformed of two tiles is finite up to translation. 
\begin{theorem}[see \cite{robinson2004symbolic}, Theorem 2.9]
	If every tiling in $\mathfrak{T}_{\A}$ has f.l.c., then the full-tiling endowed with \textup{Dist} is a compact metric space. Also, the map $\phi : \R^d \times \mathfrak{T}_\A \longrightarrow \mathfrak{T}_\A$, $\phi({\bs},\bX) = \phi_{{\bs}}(\bX) = \bX-{\bs} = \{T-{\bs}\in\bX\,|\,T\in\bX \}$ is continuous.
\end{theorem}

The continuous action of $\R^d$ on the full-tiling, defined by $\phi$ in the theorem above, is called \textbf{translation}. If $\mathfrak{X}$ is a $\phi$-invariant and closed subset of $\mathfrak{T}_{\A}$, the pair $(\mathfrak{X},\phi)$ will be called (topological) \textbf{tiling dynamical system}. As usual, we can endorse $\mathfrak{X}$ with its Borel $\sigma$-algebra $\mathcal{B}(\mathfrak{X})$, and a Borel $\phi$-invariant probability measure $\mu$, i.e., $\mu(\phi_{{\bs}}(B)) = \mu(B)$, for every ${\bs}\in\R^d$ and $B \in \mathcal{B}(\mathfrak{X})$, in order to study the measurable dynamical system $(\mathfrak{X},\mathcal{B}(\mathfrak{X}),\mu,\phi)$.

\subsection{Substitutions}
Let $L : \R^d \longrightarrow \R^d$ be a linear \textbf{expansive} map, i.e., every eigenvalue of $L$ lies outside of the closed unit disk. Now we define a substitution in the context of tilings of $\R^d$. Let $\A = \{T_1,\dots,T_m\}$ be an alphabet.
\begin{definition}
	 A \textbf{substitution} with expansion $L$ over the alphabet $\A$ is a map $\zeta : \A \longrightarrow \mathcal{P}_\A$, that satisfies  $\text{supp}(\zeta(T)) = L \left( \text{supp}(T) \right)$, for every $T \in \A$. 
\end{definition}
Simply put, $L(\text{supp}(T))$ man be split into non-overlapping tiles that are translations of those in $\A$. The definition of the substitution may be extended in a natural way to translates of the prototiles by $\zeta(T+{\bs}) = \zeta(T) + L({\bs})$, and to patches and tilings by $\zeta(\boldsymbol{P}) = \bigcup_{T\in\boldsymbol{P}} \zeta(T)$ and $\zeta(\bX) = \bigcup_{T\in\bX} \zeta(T)$ respectively. In particular, now the iterations $\zeta^n = \zeta \circ \dots \circ \zeta$ are well-defined. Associated to a substitution $\zeta$, the \textbf{substitution-matrix} denoted by $\mathcal{S}_{\zeta}$ is the nonnegative integer matrix of dimension $m\times m$ whose $(i,j$)-entry is given by the number of translates of $T_i$ that appear in $\zeta(T_j)$. If $\mathcal{S}_{\zeta}$ is a primitive matrix, we say that $\zeta$ is \textbf{primitive}.
\begin{definition}
	Let $\zeta$ be a substitution over the alphabet $\A = \{T_1,\dots,T_m\}$. The \textbf{tiling space associated to} $\zeta$ , which we will denote by $\mathfrak{X}_{\zeta}$, the set of tilings $\bX \in \mathfrak{T}_{\A}$ such that every patch of $\bX$ is a translate of a sub-patch of $\zeta^n(T_j)$, for some $1\leq j\leq m$ and $n\in\N$.
\end{definition}
A classic result in the literature ensures that for any primitive substitution $\zeta$, there exists a tiling $\bX\in\mathfrak{X}_{\zeta}$ such that $\zeta^k(\bX)=\bX$, for some $k\in\N$ (see \cite{robinson2004symbolic}, Theorem 5.10). In the case that $k=1$, the tiling $\bX$ is called a \textbf{self-affine}. If the expansion $L$ satisfies
\begin{equation*}
	\norm{L\bx} = \lambda \norm{\bx}, \quad \forall \bx\in\R^d
\end{equation*}
for some $\lambda > 1$, we will say $\bX$ is \textbf{self-similar}. When $\zeta$ admits a self-affine/self-similar tiling, we say this substitution is self-affine/self-similar.
\\

Given an alphabet $\A = \{T_1,\dots,T_m\}$ and a substitution $\zeta$ with expansion $L$, consider the \textbf{tiles of order} $k$, for $k\in\Z$, defined by $L^k\A = \{L^kT_1,\dots,L^kT_m\}$, where $L^kT_j = (L^k(\text{supp}(T_j)),L^k(\text{punc}(T_j)))$. In this new alphabet, we define the substitution $L^k\zeta$, as $L^k\zeta(L^kT_j) = L^k(\zeta(T_j))$. The \textbf{subdivision map}
\[
\Upsilon_k : \mathfrak{X}_{L^k\zeta} \longrightarrow \mathfrak{X}_{L^{k-1}\zeta},
\]
acts by dividing each tile of order $k$ into tiles of order $k-1$ according to the rule implicit in the substitution. The theorem below gives the condition necessary to define the inverse of the subdivision map. We say that a tiling $\bX \in \mathfrak{X}_{\zeta}$ is \textbf{aperiodic} if $\bX-{\bs} = \bX$ for ${\bs}\in\R^d$ only when ${\bs} = {\bzero}$. If every $\bX\in \mathfrak{X}_{\zeta}$ is aperiodic, we say that the substitution $\zeta$ is aperiodic.
\begin{theorem}[B. Solomyak, \cite{SolUniq} Theorem 1.1]\label{thm:Uniq}
	Let $\zeta$ be a primitive and aperiodic substitution with expansion $L$. Then, for every $\bX, \bY \in \mathfrak{X}_{L\zeta}$
	\[
	\Upsilon_1(\bX) = \Upsilon_1(\bY) \iff \bX = \bY.
	\]
	In particular, $\zeta$ viewed as a map from $\mathfrak{X}_{\zeta}$ to itself, is injective if and only if at least one $\bX$ is aperiodic (by primitivity).
\end{theorem}
\begin{definition}
	A tiling $\bX$ of $\R^d$ is \textbf{repetitive} if for any sub-patch $\boldsymbol{P}$ of $\bX$ there exists $R>0$ such that for any ball $B(\bx,R)$ of $\R^d$ of radius $R$, there exists a translate of $\boldsymbol{P}$ contained in $B(\bx,R)$.
\end{definition}

\begin{prop}[\cite{praggastis1999numeration}, Proposition 1.2]\label{prop:rep}
	Let $\zeta$ be a primitive substitution and such that all tilings $\mathfrak{X_{\zeta}}$ have f.l.c. Then, all tilings in $\mathfrak{X_{\zeta}}$ are repetitive.
\end{prop}

We finish this subsection recalling the ergodic properties of the tiling systems arising from an aperiodic and primitive substitution. There exists a natural $\phi$-invariant measure $\mu$ on $(\mathfrak{X}_\zeta, \mathcal{B}(\mathfrak{X}_\zeta))$ defined by the frequency of patches on any tiling of the space. In fact, this is the only invariant measure and we summarize this fact in the next result.
\begin{theorem}[see \cite{lee2003consequences}, Theorem 4.1]
	Let $\zeta$ be an aperiodic and primitive substitution. The system $(\mathfrak{X}_\zeta,\phi)$ is uniquely ergodic, and its unique ergodic measure class is given by $\mu$.
\end{theorem}

\subsection{Spectral theory}
The spectral theory of a second-countable locally compact abelian group acting on a compact metric space is well-understood from the theory of representations, see for example \cite{FKSpectral}. Here we content ourselves with only recall the spectral theorem in our particular case. We use the notation $e[x] = e^{2\pi ix}$, for $x\in\R$; and we denote by $\bx\cdot\boldsymbol{y}$ the pairing in $\C^d$ given by $\bx\cdot\boldsymbol{y} = x_1y_1 + \dots + x_dy_d$, for $\bx,\boldsymbol{y}\in\C^d$
\begin{theorem}[see \cite{einsiedler2017functional}, Theorem 9.58]
	Let $f,g \in L^2(\mathfrak{X}_\zeta,\mu)$. There is a Borel $\sigma$-finite measure $\sigma_{f,g}$ supported in $\R^d$ called \textbf{spectral measure} associated to $f,g$ (we denote also $\sigma_f = \sigma_{f,f}$) defined by its Fourier transform 
	\[
	\widehat{\sigma}_{f,g}({\bs}) = \int_{\R^d} e[-\om\cdot {\bs}] d\sigma_{f,g}(\om) = \langle f \circ \phi_{{\bs}}, g \rangle_{L^2(\mathfrak{X}_\zeta,\mu)},
	\]
	for all ${\bs} \in \R^d$.
\end{theorem}
Recall that for a self-affine tiling $\bX_0$, the boundaries of all tiles together have zero Lebesgue measure (\cite{praggastis1999numeration}, Proposition 1.1). Consider the next partition (in measure) of the tiling space.
\small{
\begin{align*}
\mathfrak{X}_{\zeta} &= \bigcup^m_{i=1} \mathfrak{X}_{T_i}, \text{ where}
\\
\mathfrak{X}_{T_i} &= \left\{ \: \phi_{{\bs}}(\bX) \in \mathfrak{X}_{\zeta} \:|\: T_i \in \bX, \: {\bs} \in \text{supp}(T_i) \right\}.
\end{align*}}

A non-trivial function $f \in L^2(\mathfrak{X}_{\zeta},\mu)$ is an \textbf{eigenfunction} if there is some $\om\in\R^d$ such that for every ${\bs}\in\R^d$ 
\[
f(\phi_{{\bs}}(\bX)) = e[{\bs}\cdot\om]f(\bX) \quad \mu\text{-a.s. for } \bX\in\mathfrak{X}_{\zeta}.
\]
In this situation, $\om$ is called an \textbf{eigenvalue} of the system $(\mathfrak{X}_{\zeta},\mu,\phi)$. If the constant functions are the only eigenfunctions, the system is said to be \textbf{weakly-mixing}. Equivalently, in terms of spectral measures, the system is weakly-mixing if and only if no spectral measure has an atom except from the trivial one at $\om = \bf{0}$. In Section 3, we recall some conditions over the expansion that ensure the system is weakly-mixing.

\subsection{Twisted Birkhoff integrals}
Here we recall the main constructions that will be used in the next section to find the modulus of continuity of the spectral measures. These concepts are multi-dimensional analogs of the ones found in \cite{bufetov2014modulus}. Consider in what follows a substitution $\zeta$ with expansion $L$ over the alphabet $\A = \{T_1,\dots,T_m\}$. We denote by $C_r$ the cube $[-r,r]^d\subseteq\R^d$, for $r>0$; and more generally by $C^{\om}_r$ the set $[\omega_1-r,\omega_1+r]\times\dots\times[\omega_d-r,\omega_d+r]$, for $\om = (\omega_1,\dots,\omega_d)\in\R^d$.
\begin{definition}
	Let $f \in L^2(\mathfrak{X}_\zeta,\mu)$, $\bX\in \mathfrak{X}_\zeta$ and $\om \in \R^d$. The \textbf{twisted Birkhoff integral} is defined as
	\[
	S^f_R(\bX, \om) = \int_{C_R} e[\om \cdot {\bs}] f \circ \phi_{{\bs}}(\bX) d{\bs}.
	\]
	We will also consider the quadratic mean of $S^f_R(\bX, \om)$:
	\[
	G_R(f,\om) = \dfrac{1}{R^d}\norm{S^f_R(\cdot, \om)}^2_{L^2(\mathfrak{X}_\zeta,\mu)} = \dfrac{1}{R^d} \int_{\mathfrak{X}_{\zeta}} S^f_R(\bY, \om)\overline{S^f_R(\bY, \om)} d\mu(\bY).
	\]
\end{definition}

\subsection{Relation between $\sigma_f$ and $S^f_R(\bX, \om)$}
The following lemma asserts that uniformly controlling the growth of twisted Birkhoff sums across the space enables a control of the spectral measure's growth. We denote the cube $[\omega_1-r,\omega_1+r]\times\dots\times[\omega_d-r,\omega_d+r] \subseteq \R^d$ by $C^{\om}_r$.
\begin{lemma}\label{lemmma: spectralmeasure}
	Let $f \in L^2(\mathfrak{X}_\zeta,\mu)$ and $\om = (\omega_1,\dots,\omega_d) \in \R^d$. Let $\Omega:[0,1) \longrightarrow \R$ be a non-decreasing function with $\Omega(0)=0$ and such that, for some constants $C$ and $R_0$, $G_R(f,\om) \leq CR^d\Omega(1/R)$ for every $R\geq R_0$. Then for all $r \leq 1/2R_0$, we have the bound
	\[
	\sigma_f(C^{\om}_r) \leq \dfrac{\pi^{2d} C}{4^d}\Omega(2r).
	\]
\end{lemma}
\begin{proof}
	By the spectral theorem (in the third equality below), we have
	{\small
		\begin{align*}{ \small}
		G_R(f,\om) &= \dfrac{1}{R^d} \langle S^f_R(\cdot,\om),S^f_R(\cdot,\om)\rangle_{L^2(\mathfrak{X}_\zeta,\mu)}  
		\\
		&= \dfrac{1}{R^d} \int_{C_R} \int_{C_R} e[\om \cdot ({\bs}-\boldsymbol{t})] \langle f \circ \phi_{{\bs}-\boldsymbol{t}},f \rangle_{L^2(\mathfrak{X}_\zeta,\mu)} d{\bs}\:d\boldsymbol{t}
		\\
		&= \dfrac{1}{R^d} \int_{C_R} \int_{C_R} e[\om \cdot ({\bs}-\boldsymbol{t})] \int_{\R^d} e[-\boldsymbol{\lambda}\cdot({\bs}-\boldsymbol{t})] d\sigma_f(\boldsymbol{\lambda}) d{\bs}\:d\boldsymbol{t}
		\\
		&= \int_{\R^d} \left[ \int_{C_R} \dfrac{1}{R^d} \int_{C_R} e[(\om-\boldsymbol{\lambda}) \cdot ({\bs}-\boldsymbol{t})] d{\bs}\:d\boldsymbol{t} \right] d\sigma_f(\boldsymbol{\lambda})
		\\
		&= \int_{\R^d} \mathfrak{F}^d_R(\omega_1 - \lambda_1,\dots,\omega_d - \lambda_d) d\sigma_f(\boldsymbol{\lambda}),
		\end{align*}
	}where $\mathfrak{F}^d_R(\omega_1 - \lambda_1,\dots,\omega_d - \lambda_d) = \prod^d_{i=1} \mathfrak{F}_R(\omega_i - \lambda_i)$ is the $d$-dimensional Fej\'er kernel and $\mathfrak{F}_R(y) = \dfrac{1}{R} \left[\dfrac{\sin(\pi R y)}{\pi y}\right]^2$. It is easily shown that if for each $i=1,\dots,d$ holds $\abs{\omega_i - \lambda_i} \leq 1/2R$, then $\mathfrak{F}^d_R(\om-\boldsymbol{\lambda}) \geq (4R)^d/\pi^{2d}$. Finally, we set $r = 1/2R \leq 1/2R_0$ to get
	{ \small
		\begin{align*}
		\sigma_f(C^{\om}_r) &\leq \dfrac{\pi^{2d}}{(4R)^d} \int_{\R^d} \mathfrak{F}^d_R(\omega_1 - \lambda_1,\dots,\omega_d - \lambda_d) d\sigma_f(\boldsymbol{\lambda})
		\\
		&= \dfrac{\pi^{2d}}{(4R)^d} G_R(f,\om) \leq \dfrac{\pi^{2d} C}{4^d} \Omega(2r).
		\end{align*}
	}
\end{proof}
The class of functions we will work with are defined next. 
\begin{definition}
	We will say a function $f:\mathfrak{X_{\zeta}} \longrightarrow \C$ is a \textbf{cylindrical function} if it is integrable and only depends on the tile containing ${\bzero}$. More precisely, there exist integrable $\psi_i: \R^d \longrightarrow \C$, for $i=1,\dots,m$ such that $\text{supp}(\psi_i)\subseteq T_i$, $\psi_i \equiv 0$ on the border of $T_i$ and 
	\small{
	\begin{equation}\label{cylindrical}
		f(\bX) = \sum_{i=1}^m \psi_i({\bs}) \mathds{1}_{\mathfrak{X}_{T_i}}(\bX), \quad \text{if } T_i + {\bs} \in \bX,\, {\bs}\in T_i.
	\end{equation}}We will say $f$ is Lispchitz (resp. smooth) if all $\psi_i$ in (\ref{cylindrical}) are Lispchitz (resp. smooth).
\end{definition}
\begin{remark}
	Cylindrical functions are in fact what are called cylindrical functions of order zero, since we could more generally consider indicator functions of tiles of higher order in (\ref{cylindrical}). As remarked in \cite{bufetov2013limit}, when we consider cylindrical functions of all orders, we obtain a dense subspace of $L^2(\mathfrak{X}_\zeta,\mu)$. This would allow us to study the whole family of spectral measures by approximation, as it is done for example in \cite{JEP_2021__8__279_0}. For simplicity we will only work with cylindrical functions of order zero as in the previous definition.
\end{remark}
We now analyze the twisted integral for a cylindrical function of the form $f=\psi \mathds{1}_{\mathfrak{X}_{T_i}} $. Let $]C^{\bzero}_R[^{\bX} = \{X_1,\dots,X_N\}$ and $\bx_j = \text{punc}(X_j)$. For two tiles $T$ and $S$, define $\delta_{T,S}=1$ if $T$ is a translate of $S$, and $\delta_{T,S}=0$ otherwise. Then,
{ \small
\begin{align}
S^f_R(\bX, \om) &= \int_{\text{supp}(]C_R[^{\bX})} e[\om \cdot {\bs}] f(\phi_{\bs}(\bX)) d{\bs} + \underbrace{\int_{C_R\setminus \text{supp}(]C_R[^{\bX})} e[\om \cdot {\bs}] f(\phi_{\bs}(\bX)) d{\bs}}_{=\Delta(R)} \label{error}
\\
&= \sum^{N}_{j=1} \delta_{T_i,X_j} \int_{\text{supp}(X_j)} e[\om \cdot {\bs}] f(\phi_{\bs}(\bX)) d{\bs} + \Delta(R)
\\
&= \sum^{N}_{j=1} \delta_{T_i,X_j} \int_{\text{supp}(X_j) - \bx_j} e[\om \cdot ({\bs}+\bx_j)]f(\phi_{{\bs} + \bx_j}(\bX))d{\bs} + \Delta(R)
\\
&= \sum^{N}_{j=1} \delta_{T_i,X_j} \int_{\text{supp}(T_i)} e[\om \cdot ({\bs}+\bx_j)]\psi({\bs})d{\bs} + \Delta(R)
\\
&=\int_{\text{supp}(T_i)} e\left[\om \cdot {\bs}\right]\psi({\bs})d{\bs}\: \sum^{N}_{j=1} \delta_{T_i,X_j} e\left[\om \cdot \bx_j\right] + \Delta(R)
\\
&= \widehat{\psi}(\om) \Phi_i \left(\: ]C_R[^{\bX},\om \right) + \Delta(R),\label{fourier}
\end{align} }
where $\widehat{\psi}$ denotes the Fourier transform of $\psi$. 

The sum $\Phi_i \left(\: ]C_R[^{\bX},\om \right) = \sum^{N}_{j=1} \delta_{T_i,X_j} e\left[\om \cdot \bx_j\right]$ is called \textbf{structure factor} (see \cite{bufetov2014modulus} and references therein). It can also be defined for any patch $\boldsymbol{P}$ instead of $]C_R[^{\bX}$. Note that by definition $\abs{\Phi_i(\boldsymbol{P},\om)} = \abs{\Phi_i(\boldsymbol{P}+{\bs},\om)}$, for every ${\bs}\in \R^d$. Next, we examine the behavior of the structure factor for tiles of order $n$, treated as a patch. If $\zeta(T_j) = \bigcup_{1\leq k \leq m} \bigcup_{1\leq l \leq n^{j}_{k}} \left\{ T_k + {{\bs}}^{k}_{l}(j) \right\}$, then
\[
\zeta^n(T_j) = \bigcup_{1\leq k \leq m} \bigcup_{1\leq l \leq n^{j}_{k}} \left\{ \zeta^{n-1} (T_k) + L^{n-1}({{\bs}}^{k}_{l}(j)) \right\},
\]
where $n^j_k = \mathcal{S}_{\zeta}(k,j)$ and $n^j_1 + \dots + n^j_m = N^j = \#\zeta(T_j)$. By definition,
\begin{align}\label{eq:cocycle}
	\Phi_i(\zeta^n(T_j),\om) &= \sum_{k=1}^{m} \sum_{l=1}^{n^j_k} e\left[\om \cdot L^{n-1}({{\bs}}^{k}_{l}(j)) \right] \Phi_i(\zeta^{n-1}(T_k),\om).
\end{align}
Define
\[
\Pi_n(\om) (k,i) = \Phi_i(\zeta^n(T_k),\om),\quad \mathcal{M}_{n-1}(\om) (j,k) = \sum_{l=1}^{n^j_k} e\left[\om \cdot L^{n-1}({{\bs}}^{k}_{l}(j)) \right].
\]
Equality (\ref{eq:cocycle}) can be expressed as $\Pi_n(\om) = \mathcal{M}_{n-1}(\om)\Pi_{n-1}(\om)$. Note also that $\Pi_0(\om) = I_{m}$, the identity matrix of dimension $m$. Thus,
\[
\Pi_n(\om)= \mathcal{M}_{n-1}(\om) \dots \mathcal{M}_1(\om).
\]
\begin{remark}
	The implicit cocycle defined by the matrices $\Pi_n$ is similar to the one discussed in \cite{bufetov2014modulus}. A more general version is defined in \cite{solomyak2024spectral}, but we will not need such generality.
\end{remark}

\section{Preliminary results}
In this section we prove most of the facts we need to prove the regularity of the spectral measures in the case of strongly non-Pisot substitution tilings (definition after Theorem \ref{wmcondition}. For the rest of this section we fix:

\begin{itemize}
	\item $\A = \{ T_1,\dots,T_m \}$ a tile alphabet satisfying $\text{punc}(T_j) = \bf{0}$.
	\item $\zeta$ an aperiodic, primitive and self-affine substitution over $\A$, having as expansion the linear transformation $L:\R^d \longrightarrow \R^d$.
	\item $(\mathfrak{X}_\zeta, \mathcal{B}(\mathfrak{X}_\zeta),\mu,\phi)$ the dynamical system associated to $\zeta$.
\end{itemize}

\subsection{Tower structure}
To obtain estimates on the growth of twisted Birkhoff integrals, we will take advantage of the tower structure of substitution tiling systems, which we recall in the next paragraph. The aperiodicity assumption on the substitution is crucial, because only in that situation we may decompose an arbitrary patch of a tiling as a unique union of tiles of order $\geq 1$.

For a tiling $\bX \in \mathfrak{X}_{\zeta}$ we set $\bX^{(0)} = \bX$ and $\bX^{(k)} = \Upsilon^{-1}_{k-1}\circ\dots\circ\Upsilon^{-1}_{0}(\bX)$. By Theorem \ref{thm:Uniq} we can view $\bX$ as a tiling on the alphabet $L^k\A$, for any $k\in\Z$. Actually, for a patch of $\bX$ we may write $\boldsymbol{P} = \{T \in \mathcal{R}^k(\boldsymbol{P}) \: | \: 0\leq k \leq n \}$, where $n$ is the first integer such that $]\text{supp}(\boldsymbol{P})[^{\bX^{(n+1)}} = \emptyset$, and the sets $\mathcal{R}^k(\boldsymbol{P})$ are inductively defined by
\[
\mathcal{R}^n(\boldsymbol{P}) =\: ]\text{supp}(\boldsymbol{P})[^{\bX^{(n)}} \text{ and } \mathcal{R}^k(\boldsymbol{P}) =\: ]\text{supp}(\boldsymbol{P})[^{\bX^{(k)}}\: \setminus\: ]\text{supp}(\mathcal{R}^{k+1}(\boldsymbol{P}))[^{\bX^{(k)}},
\]
for $0\leq k < n$. Simply put, $\mathcal{R}^k(\boldsymbol{P})$ is the set of order-$k$ tiles within  $\text{supp}(\boldsymbol{P})$ that are not part of any order-$(k+1)$ tile fully contained in $\text{supp}(\boldsymbol{P})$. Denote by $\norm{L}$ the operator norm of $L$ induced by the norm $\norm{\cdot}_\infty$ on $\R^d$, and by $C_R$ the cube $[-R,R]^d$, for $R>0$.

\begin{lemma}\label{thm:Factor}
	Let $f = \mathds{1}_{\mathfrak{X}_{T_i}}$ and $\om \in \R^d$. Assume there exists $1 < \xi < \norm{L}$, and a sequence $(F_{\om}(n))_{n\geq 0}$ such that for every $T_j \in \A$ we have
	\begin{itemize}
		\item $ \dfrac{F_{\om}(n)}{\xi} \leq F_{\om}(n+1) \leq F_{\om}(n)$.
		\item $\abs{\Phi_i(\zeta^n(T_j),\om)} \leq \#\zeta^n(T_j) \: F_{\om}(n)$.\\
	\end{itemize}
	Then for all $\bX \in \mathfrak{X}_{\zeta}$ and $R\geq2$,
	\[
	\abs{S^f_R(\bX, \om)} = \mathcal{O}\left( R^{d}F_{\om}(\floor{\log_{\norm{L}}(R)})\right)  + \mathcal{O}(R^{d-1}).
	\]
\end{lemma}
\begin{proof}
	By Theorem \ref{thm:Uniq} we have $]C_R[^{\bX} = \{T \in \mathcal{R}^k(]C_R[^{\bX}) \: | \: 0\leq k \leq n \}$. Let $D_{\text{max}}$ be the maximum diameter of a prototile's support and $V_{min}$ the minimum volume of these. Denote by $U(\partial C_R, r)$ the set of points which are at a distance less or equal than $r$ from $\partial C_R$, the border of $C_R$. It is easy to see that $C_R \setminus \text{supp}(]C_R[^{\bX}) \subseteq U(\partial C_R, D_{max})$. In terms of the Lebesgue measure of $\R^d$, which we denote by $\mathcal{L}^d$, we can deduce
	{ \small
	\begin{align*}
	\mathcal{L}^d \left( C_R \setminus \text{supp}(]C_R[^{\bX}) \right) \leq \mathcal{L}^{d}(U(\partial C_R, D_{max}))
	&\leq (2R + 2D_{max})^d - (2R - 2D_{max})^d
	\\
	&= \mathcal{O}( R^{d-1} ),
	\end{align*}
	}that is, the error term in (\ref{error}), satisfies $\Delta(R) = \int_{C_R \setminus \; \text{supp}(]C_R[^{\bX})} e[\om\cdot {\bs}] \mathds{1}_{\mathfrak{X}_{T_i}}(\bX-{\bs}) ds = \mathcal{O}(R^{d-1})$. Since $\text{supp}(T_i)$ has finite volume, we may estimate the Birkhoff sum $\abs{S^f_R(\bX, \om)}$ as follows.
	{ \small
	\begin{align*}
	\abs{S^f_R(\bX, \om)} &\leq \left| \int_{\text{supp}(T_i)} e[ \om \cdot {\bs}]d{\bs} \right| \abs{\Phi_i \left(\: ]C_R[^{\bX},\om \right)}  + \abs{\Delta(R)}
	\\
	&= \mathcal{O}\left( \sum^{n}_{k=0} \sum^{\#\mathcal{R}^k (]C_R[^{\bX})}_{l=1} \abs{\Phi_i(\zeta^k (X^k_l) + \bx^k_l, \om)} \right) + \mathcal{O}(R^{d-1}); \: X^k_l \in \A,\, \bx^k_l \in \R^d
	\\
	&= \mathcal{O}\left(\sum^{n}_{k=0} \#\mathcal{R}^k (]C_R[^{\bX}) \max_{j=1,\dots,m} \abs{\Phi_i(\zeta^k (T_j), \om)}\right) + \mathcal{O}(R^{d-1})
	\\
	&= \mathcal{O}\left(\sum^{n}_{k=0} \#\mathcal{R}^k (]C_R[^{\bX}) \max_{j=1,\dots,m} \#\zeta^k( T_j) \: F_{\om}(k) \right) + \mathcal{O}(R^{d-1}),
	\end{align*}
	}where the last inequality follows from the second hypothesis on the sequence $(F_{\om}(n))_{n\geq 0}$. Since $\zeta$ is primitive, the growth of $\max_{j=1,\dots,m} \#\zeta^k (T_j)$ is controlled by the Perron–Frobenius eigenvalue of $\mathcal{S}_{\zeta}$ denoted by $\theta$ (in fact, $\theta= \det(L)$). By the Perron–Frobenius theorem, there exist constants $c_1$, $c_2 >0$ such for every $k\geq0$ 
	\[
	c_1 \theta^k \leq \min_{j=1,\dots,m} \#\zeta^k (T_j) \leq \max_{j=1,\dots,m} \#\zeta^k (T_j) \leq c_2 \theta^k.
	\]
	
	Using the first assumption on $F_{\om}(n)$ we get
	{ \small
	\begin{align*}
	\abs{S^f_R(\bX, \om)} &= \mathcal{O}\left( \: \sum^{n}_{k=0} \#\mathcal{R}^k (]C_R[^{\bX}) \max_{j=1,\dots,m} \#\zeta^k T_j \: F_{\om}(k) \right) + \mathcal{O}(R^{d-1})
	\\
	&= \mathcal{O}\left(\: \sum^{n}_{k=0} \#\mathcal{R}^k (]C_R[^{\bX}) \theta^k \: F_{\om}(k) \right) + \mathcal{O}(R^{d-1})
	\\
	&= \mathcal{O}\left(\: \sum^{n}_{k=0} \#\mathcal{R}^k (]C_R[^{\bX}) \theta^k \: \xi^{n-k} F_{\om}(n) \right) + \mathcal{O}(R^{d-1})
	\\
	&= \mathcal{O}\left( \xi^{n} F_{\om}(n)\: \sum^{n}_{k=0} \#\mathcal{R}^k (]C_R[^{\bX}) \theta^k \xi^{-k}\: \right)  + \mathcal{O}(R^{d-1}).
	\end{align*} 
	}
	
	Next, we repeat the argument used to bound $\Delta(R)$ to deal with $\#\mathcal{R}^k(]C_R[^{\bX})$: since $\text{supp}\left( \mathcal{R}^k \left( ]C_R[^{\bX} \right) \right) \subseteq U(\partial C_R,D_{max}\norm{L}^{k+1})$, then
	\small{
	\begin{align*}
	\#\mathcal{R}^k \left( ]C_R[^{\bX} \right) &\leq \dfrac{ \mathcal{L}^{d}(U(\partial C_R, D_{\text{max}} \norm{L}^{k+1})) }{V_{\text{min}} \theta^{k}} 
	\\
	&\leq \dfrac{1}{V_{min}\theta^{k}} \left[(2R + 2D_{max}\norm{L}^{k+1})^d - (2R - 2D_{max}\norm{L}^{k+1})^d\right].
	\end{align*}}It follows that
	\small{
	\begin{align*}
	\abs{S^f_R&(\bX, \om)} \\
	&= \mathcal{O}\left( \xi^{n} F_{\om}(n)\: \sum^{n}_{k=0} \left[(2R + 2D_{max}\norm{L}^{k+1})^d - (2R - 2D_{max}\norm{L}^{k+1})^d\right] \xi^{-k} \right)
	\\
	& \quad + \mathcal{O}(R^{d-1})
	\\
	&= \mathcal{O}\left( \xi^{n} F_{\om}(n)\: \sum^{n}_{k=0} \dfrac{\norm{L}^{kd}}{\xi^{k}} \left[\left(\dfrac{2R}{\norm{L}^{k}} + 2D_{max}\norm{L}\right)^d - \left(\dfrac{2R}{\norm{L}^{k}} - 2D_{max}\norm{L}\right)^d\right] \right) \\
	& \quad + \mathcal{O}(R^{d-1}) 
	\\
	&= \mathcal{O}\left( \xi^{n} F_{\om}(n)\: \sum^{n}_{k=0} \dfrac{\norm{L}^{kd}}{\xi^{k}} \dfrac{ (2R)^{d-1} }{\norm{L}^{k(d-1)} } \right) + \mathcal{O}(R^{d-1}) \\
	&= \mathcal{O}\left( R^{d-1} F_{\om}(n) \norm{L}^{n} \right) + \mathcal{O}(R^{d-1}).
	\end{align*}}
	
	Let $D_{min}$ be the diameter of the largest ball that fits within any prototile. Then, because there is at least one tile of order $n$ with support contained in $C_R$, necessarily the diameter of $C_R$ exceeds $D_{min}\norm{L}^n$, i.e., $D_{min}\norm{L}^n \leq 2\sqrt{d}R$, which implies $\norm{L}^{n} = \mathcal{O}(R)$. On the other hand, since there are no tiles of order $n+1$ whose support is contained in $C_R$, then necessarily $R \leq D_{max}\norm{L}^{n+1}$. Rearranging this we deduce there exists a constant $C>0$ depending only on the substitution such that $\floor{\log_{\norm{L}}(R) - C} \leq n$. Introducing this in the last inequality above allow us to conclude:
	{ \small
	\begin{align*}
	\abs{S^f_R(\bX, \om)} &= \mathcal{O}\left( R^{d-1} F_{\om}(n) \norm{L}^{n}\right) + \mathcal{O}(R^{d-1})
	\\
	&= \mathcal{O}\left( R^{d-1} F_{\om}(\floor{\log_{\norm{L}}(R) - C}) R  \right) + \mathcal{O}(R^{d-1}) \\
	&= \mathcal{O}\left( R^d F_{\om}(\floor{\log_{\norm{L}}(R)}) \right) + \mathcal{O}(R^{d-1}).
	\end{align*} }
\end{proof}

\subsection{Bound for the structure factor}
The following proposition estimates the structure factor of tiles of order $n$ using the matrix Riesz product from Section 2.5and reducing it to a scalar Riesz product, as in \cite{bufetov2014modulus}. For $\bX \in \mathfrak{X_{\zeta}}$, denote by
\[
\mathfrak{R}(\bX)= \{{\bv}\in\R^d, {\bv}\neq 0 \,|\, \exists\, X_1,X_2 \in \bX: X_1 = X_2 + {\bv} \}
\]
the set of \textbf{return vectors}. Note that $\mathfrak{R}(\bX)$ does not depend on the tiling $\bX$, so we may also denote this set by $\mathfrak{R}(\zeta)$. By the primitivity assumption on the substitution, after passing to a suitable power of it, we may suppose that each patch of the form $\zeta(T_j)$ contains some fixed return vector ${\bv}$. We will call such a vector a \textbf{good return vector}. The (finite) set of good return vectors will be denoted by $\mathfrak{GR}(\zeta)$. Using repetitivity of tilings (Proposition \ref{prop:rep}), we can assume the set of good return vectors contains a basis of $\R^d$. In fact, we have 
\small{
\[
\overline{\left\{ \dfrac{{\bv}}{\norm{{\bv}}}_2 \middle|\, {\bv} \in \mathfrak{R}(\bX) \right\} } = S^{d-1} \text{ (the unit sphere), }
\]
}this is proved in \cite{solomyak1997dynamics}, proof of Theorem 4.4. Let $\mathfrak{B} = \{{\bv}_1,\dots,{\bv}_d\}$ a basis of $\R^d$ made of good return vectors. For every $n\geq 1$ there exists ${\bv}_{k_n} \in \mathfrak{B}$ such that the angle
\begin{equation}\label{rm:infty}
	\angle((L^{\sf T})^n \om, {\bv}_{k_n}) > \delta(\mathfrak{B}) > 0,
\end{equation}
where $\delta = \delta(\mathfrak{B})$ is a positive constant depending only on the basis $\mathfrak{B}$, ergo, on the substitution $\zeta$. In particular, $\abs{(L^{\sf T})^n \om \cdot {\bv}_{k_n}} \to \infty$ as $n\to \infty$.\\

For $\tau \in \R$, we denote the distance from $\tau$ to the nearest integer by $\norm{\tau}_{\R/\Z}$.

\begin{prop}\label{thm:ProdBound}
	There exists $\mathfrak{m} \in (0,1)$ depending only on the substitution $\zeta$ such that for every $T_i,T_j \in \A$ and $\om \in \R^d$:
	\small{
	\[
	\abs{\Phi_i(\zeta^n (T_j),\om)} = \mathcal{O}\left( \#\zeta^n (T_j) \prod^{n-1}_{l=0} 1 - \max_{{\bv}\in\mathfrak{GR}(\zeta) } \mathfrak{m}\norm{\om \cdot L^l{\bv}}^2_{\R/\Z} \right).
	\]}
\end{prop}
\begin{proof}
	This proof follows Proposition 3.5 in \cite{bufetov2014modulus} but differs by selecting an appropriate good return vector at each step. Below we write $\mathcal{S} = \mathcal{S}_\zeta$ and $\om$ is omitted of the matrices $\mathcal{M}_n(\om)$ and $\Pi_n(\om)$. First, we estimate the absolute value of  $\mathcal{M}_n(j,k)$. As in subsection 2.5, we write 
	\small{
	\begin{equation}\label{decomposition}
		\zeta^n(T_j) = \bigcup_{1\leq k \leq m} \bigcup_{1\leq l \leq n^{j}_{k}} \left\{ \zeta^{n-1} (T_k) + L^{n-1}({{\bs}}^{k}_{l}(j)) \right\}.
	\end{equation}}

	Let ${\bv}_{k_n} \in \mathfrak{GR}(\zeta)$ be a good return vector that realizes the maximum
	\[
	\max_{{\bv}\in \mathfrak{GR}(\zeta)} \norm{\om \cdot L^n{\bv}}_{\R/\Z} =  \norm{\om \cdot L^n{\bv}_{k_n}}_{\R/\Z}.
	\]
	Let $T_{k_n}$ be a prototile such that two of its translates in $\zeta(T_j)$ are connected by ${\bv}_{k_n}$ corresponding to indices $l_1$ and $l_2$ in the decomposition (\ref{decomposition}). As the number of summands in this entry equals  $\mathcal{S}^{\sf T}(j,k_n)$, it follows that
	\small{
	\begin{align*}
	\abs{\mathcal{M}_n(j,k_n)} &\leq \mathcal{S}^{\sf T}(j,k_n) - 2 + \left| e[\om\cdot L^n{{\bs}}^{k_n}_{l_1}(j)] + e[\om\cdot L^n{{\bs}}^{k_n}_{l_2}(j)] \right|
	\nonumber\\
	&\leq \mathcal{S}^{\sf T}(j,k_n) - 2 + \Big| 1 + e[\om\cdot L^n(\underbrace{{{\bs}}^{k_n}_{l_2}(j) - {{\bs}}^{k_n}_{l_1}(j)}_{= {\bv}_{k_n}})] \Big| 
	\nonumber\\
	&\leq \mathcal{S}^{\sf T}(j,k_n) - \dfrac{1}{2}\norm{\om\cdot L^n{\bv}_{k_n}}^2_{\R/\Z},
	\end{align*}}using the inequality $\abs{1 + e[\tau]} \leq 2 - \dfrac{\norm{\tau}^2_{\R/\Z}}{2}$ in the last step. Consider $\bx \in \R^m$ with positive entries and denote by $\mathcal{N}^{\abs{\cdot}}$ the matrix whose entries are the absolute values of a matrix $\mathcal{N}$. Then,
	\small{
	\begin{align*}
	\mathcal{M}^{\abs{\cdot}}_n \bx(j) \leq \sum^m_{l=1} \abs{\mathcal{M}_n(j,l)}\bx(l) &\leq \sum^m_{l=1} \mathcal{S}^{\sf T}(j,l)\bx(l) - \dfrac{1}{2}\norm{\om\cdot L^n{\bv}_{k_n}}^2 \bx(k_n) 
	\\
	&\leq \left( 1 - \Xi[\bx]\,\norm{\om\cdot L^n{\bv}_{k_n}}^2_{\R/\Z} \right) \: \mathcal{S}^{\sf T}\bx(j),
	\end{align*}}where the constant $\Xi[\bx]$ may be taken as
	\[
	\Xi[\bx] = \dfrac{\bx(k_n)}{2m \max_{\,l,l' = 1,\dots,m} \mathcal{S}^{\sf T}(l,l') \max_{\,l=1,\dots,m} \bx(l)}.
	\]	
	
	Since $\Pi_n = \mathcal{M}_{n-1} \dots \mathcal{M}_0 \Pi_0$ and $\Pi_0^{\abs{\cdot}} = I_{m}$, then
	\[
	\Pi^{\abs{\cdot}}_n(\cdot,i) \leq \mathcal{M}^{\abs{\cdot}}_{n-1} \dots \mathcal{M}^{\abs{\cdot}}_0 \, {\bf 1},    
	\]
	where ${\bf 1}$ is the vector of $\R^d$ with all its entries equal to one. In this way, inductively we deduce the next inequality by coordinates:
	\[
	\Pi^{\abs{\cdot}}_n(\cdot,i) \leq (\mathcal{S}^{\sf T})^n{\bf 1} \, \prod^{n-1}_{l=0} 1 - \Xi\left[(\mathcal{S}^{\sf T})^l{\bf 1}\right] \norm{\om \cdot L^l{\bv}_{k_n}}^2_{\R/\Z}.
	\]
	Finally, by the Perron–Frobenius theorem $(\mathcal{S}^{\sf T})^n {\bone} \leq C \#\zeta^n(T_j){\bone}$, where $C>0$ depends only on $\zeta$. By the same theorem $\mathfrak{m} = \inf_{l\geq0} \Xi[(\mathcal{S}^{\sf T})^l{\bone}] > 0$ and this concludes the proof. 
\end{proof}
\begin{corollary}\label{thm:SumBound}
	Consider the hypotheses and notation of Proposition \ref{thm:ProdBound}, and put $f = \mathds{1}_{\mathfrak{X}_{T_i}}$. For each $\bX \in \mathfrak{X}_{\zeta}$ and $i = 1,\dots, m$
	\small{
	\begin{align*}
	\abs{S^f_R(\bX, \om)} = \mathcal{O}\left( R^{d} \prod^{\floor{\log_{\norm{L}}(R)}}_{l=0} 1 - \max_{{\bv}\in \mathfrak{GR}(\zeta)} \mathfrak{m} \norm{\om \cdot L^l{\bv}}^2_{\R/\Z} \right).
	\end{align*}}
\end{corollary}
\begin{proof}
	Let $F_{\om}(n) = \prod^{n-1}_{l=0} 1 - \max_{{\bv}\in \mathfrak{GR}(\zeta)} \mathfrak{m} \: \norm{\om \cdot L^l{\bv}}^2_{\R/\Z}$. Decreasing $\mathfrak{m}$ if necessary, assume $\mathfrak{m} \leq (\norm{L}-1)/(\norm{L}+1)$ and take $\xi  = (\norm{L}+1)/2$. The second assumption in Lemma \ref{thm:Factor} is satisfied by Proposition \ref{thm:ProdBound}, and since $1-\mathfrak{m} \geq \dfrac{2}{\norm{L}+1}$ we have
	\[
	\dfrac{F_{\om}(n)}{\xi} = \dfrac{2F_{\om}(n)}{\norm{L}+1} \leq F_{\om}(n)(1-\mathfrak{m}) \leq F_{\om}(n+1).
	\]
	It is easy to see that we can ignore the term $\mathcal{O}(R^{d-1})$ in this case.
\end{proof}
To prove Theorem \ref{thm:main}, we must bound the term $R^d$ when $\norm{\om}$ is large. This will be accomplished using the decay of the Fourier transform appearing in equation \eqref{fourier} for Lipschitz functions. We recall that for a Lipschitz function $g:\R^d \to \C$, its Lipschitz norm $\norm{g}_{\text{Lip}}$ is defined by
\[
\norm{g}_{\text{Lip}} = \norm{g}_\infty + S, \quad S = \inf \left\{s>0\,\middle|\, \abs{g(\bx) - g(\boldsymbol{y})} \leq s\abs{\bx - \boldsymbol{y}},\, \forall \bx,\boldsymbol{y} \in \R^d  \right\}.
\]
\begin{corollary}\label{cor:sumbound}
	Let $f:\mathfrak{X_{\zeta}}\to\C$ be a Lispchitz cylindrical function. Then
	\small{
	\[
	\abs{S^f_R(\bX, \om)} =  \mathcal{O}\left( \norm{f}_{\textup{Lip}} (1+\norm{\om})^{-1} R^{d} \prod^{\floor{\log_{\norm{L}}(R)}}_{l=0} 1 - \max_{{\bv}\in \mathfrak{GR}(\zeta)} \mathfrak{m} \norm{\om \cdot L^l{\bv}}^2_{\R/\Z} \right),
	\]}where $\norm{f}_{\textup{Lip}} := \max_i \norm{\psi_i}_{\textup{Lip}}$, and $\norm{\cdot}_{\textup{Lip}}$ is the Lipschitz norm.
\end{corollary}

\begin{proof}
	The proof is a direct consequence of equation \eqref{fourier}, Corollary \ref{thm:SumBound} and the next fact: the Fourier transform of the Lipschitz functions $\psi_i$ satisfy
	\begin{equation}
		\abs{\widehat{\psi}_i(\om)} = \mathcal{O}(\norm{\psi_i}_{\text{Lip}} (1+\norm{\om})^{-1}).
	\end{equation}
\end{proof}

\subsection{Spectrum of self-affine tilings expansions}\label{subsec:conditions}
Let us recall some results concerning the kind of linear transformations that are expansions of self-affine tilings of $\R^d$.

\begin{theorem}[R. Kenyon, B. Solomyak, \cite{Kenyon}, Theorem 3.1] \label{thm:SubstitutionPerron}
	Let $L : \R^d \longrightarrow \R^d$ be the expansion of a self-affine substitution that is diagonalizable (over $\C$). Then
	\begin{enumerate}
		\item Every eigenvalue $L$ is an algebraic integer.
		\item If $\lambda_1$ is an eigenvalue of multiplicity $k_1$ and $\lambda_2$ is a Galois conjugate of $\lambda_1$ of multiplicity $k_2$, then $\abs{\lambda_1} > \abs{\lambda_2}$ or $\lambda_2$ is an eigenvalue of $L$ and $k_2 \geq k_1$.
	\end{enumerate}
\end{theorem}
\begin{remark}
	An expansion satisfying the second condition is called \textbf{Perron expansion}. If we drop the assumption that the expansion is diagonalizable, the result still holds with an extended notion of Perron expansion(see \cite{SelfExpansions}).
\end{remark}

The following result gives conditions for a substitution to guarantee that the associated dynamical system is weakly-mixing. A set $\Lambda = \{\lambda_1,\dots,\lambda_b\}$ is said to be a \textbf{Pisot family} if each $\lambda_i$ is an algebraic integer outside the unit disk and such that for every of its Galois conjugates $\lambda$ which are not in $\Lambda$ it must hold $\abs{\lambda}<1$. If a family $\Lambda$ does not satisfy this property we will call it a \textbf{non-Pisot family}. Any set of algebraic integers $\Lambda$ may be decomposed as the union of subsets whose elements share the same minimal polynomial. If each of this subsets is a non-Pisot family, $\Lambda$ is called a \textbf{totally non-Pisot family}.

\begin{theorem}[see \cite{robinson2004symbolic}, Theorem 7.6]\label{wmcondition}
	Let $\zeta$ be a primitive and injective substitution with expansion $L$, which is diagonalizable and totally non-Pisot. Then $(\mathfrak{X_{\zeta}}, \phi)$ is weakly-mixing. 
\end{theorem}
\begin{definition}\label{def:stnp}
	Let $\Lambda=\Lambda_1\cup\dots \cup\Lambda_p$ be a totally non-Pisot family decomposed according to the corresponding minimal polynomials. 
	If for each $i = 1,\dots,p$ and $\lambda\in\Lambda_i$ there exists a Galois conjugate of $\lambda$ outside the closed unit disk and that is not an element of $\Lambda_i$, we will say $\Lambda$ is a \textbf{strongly totally non-Pisot family}. We will say a substitution $\zeta$ is \textbf{strongly totally non-Pisot} if the spectrum of the expansion of $\zeta$ is a strongly totally non-Pisot family.
\end{definition}

To finish this section, we recall a necessary and sufficient condition for $\om\in\R^d$ to be an eigenvalue of the system $(\mathfrak{X}_{\zeta},\mu,\phi)$.
\begin{theorem}[B. Solomyak, \cite{solomyak1997dynamics}, Theorem 5.1]\label{thm:eigenvalues}
	If $\zeta$ is a primitive aperiodic substitution with expansion $L$, then $\om\in\R^d$ is an eigenvalue of the system $(\mathfrak{X}_{\zeta},\mu,\phi)$ if and only if 
	\[
	\lim_{n \to \infty} e\left[ (L^{\sf{T}})^n\om \cdot{\bv} \right] = 1, \quad \forall \,{\bv}\in \mathfrak{R}(\zeta).
	\]
\end{theorem}

\section{Modulus of continuity for strongly totally non-Pisot substitutions}
We will assume $L$ is diagonalizable over $\C$. Until here we have not made any assumption on the location of the punctures on the prototiles. From now on we will take them as \textbf{control points}: for each prototile, choose one tile in its image by the substitution $\zeta$. The preimage of this tile is a compact subset of the original tile. Repeating this process indefinitely, we obtain a sequence of nested nonempty compact sets whose diameter goes to zero. The only point in the intersection of the sequence is the control point of the prototile, and from now on, its puncture. Since the substitution is primitive, we may assume that the chosen tile in every image is one of some specific type, say $T^*$.

Following \cite{Kenyon,lee2008pure}, let $\bX_0\in \mathfrak{X_{\zeta}}$ be a self-affine tiling and let $\mathcal{C}$ be the set of its control points. Since we took the puncture of each tile as being the control point associated to some particular tile type, we have that for any $T,S \in \bX_0 $,
\[
L(\text{punc}(T)-\text{punc}(S)) \in \mathfrak{R}(\bX_0).
\]
Consider $\mathcal{D} = \left\{ \text{punc}(T)-\text{punc}(S) \,\middle|\, T,S \in \bX_0, T\neq S, T\cap S\neq \emptyset \right\}$. Since tilings in $\mathfrak{X_{\zeta}}$ have f.l.c, $\mathcal{D}$ is a finite set. By repetitivity, modulo taking a higher power of the substitution $\zeta$, we may assume that for each $\boldsymbol{d} \in \mathcal{D}$ and $T_j\in\A$ we have $L\boldsymbol{d}\in\zeta(T_j)$. In particular, for all $\boldsymbol{d} \in \mathcal{D}$, $L\boldsymbol{d} \in \mathfrak{GR}(\zeta)$. Note also that the free abelian group generated by the set of control points $G=\langle \mathcal{C} \rangle$, is also generated by $\mathcal{D}$.

Let us fix ${\bv}_1,\dots,{\bv}_N \in \mathcal{D}$ a set of generators for $G$. 
Let $V = ({\bv}_1,\dots,{\bv}_N)$ be the $d\times N$ full rank matrix whose columns are the previously fixed generators. There exists a $N\times N$ integer matrix $M$ such that
\[
LV = VM
\]
(we keep calling $L$ the matrix representing the expansion $L$ in the canonical basis). It is proved in Lemma 3.3 of \cite{Kenyon} that $M$ is diagonalizable. For $\om\in\R^d$, consider the norm $\norm{\om}_V := \max_{i=1,\dots,N}\abs{\om\cdot{\bv}_i}$.
\begin{prop}\label{prop:spectralestimate}
	Let $\zeta$ be a strongly totally non-Pisot substitution and $f$ be a cylindrical function of $\mathfrak{X}_\zeta$ as in (\ref{cylindrical}), with all $\psi_i$ being Lipchitz. There exists $\gamma>0$ only depending on $\zeta$ such that for any $\om\neq{\bzero}$ there exist explicit constants $C=C(\zeta) > 0$ and $r_0 = r_0(\norm{\om}_V)>0$ such that
	\[
	\sigma_f (C^{\om}_r) \leq C \norm{f}_{\textup{Lip}}^2 \log(1/r)^{-\gamma} \quad \forall r < r_0.
	\]
\end{prop}

\begin{proof}
	Write
	\begin{equation}\label{decomposition}
	{\bK}_n^V + \bm{\epsilon}^V_n = V^{\sf T} (L^{\sf T})^n \om = M^n V^{\sf T} \om = M^n \widetilde{\om}, 
	\end{equation}
	where ${\bK}_n^V\in \Z^N$ is the closest integer lattice point to $ V^{\sf T} (L^{\sf T})^n \om$. By the remark just after equation (\ref{rm:infty}), consider $n_0 = n_0(\om)$ such that for all $n\geq n_0$ we have ${\bK}_n^V \neq {\bzero}$, which is guaranteed for 
	\begin{equation}
		n_0 \geq Z_1 \max(1,\log(\norm{\om}^{-1}_V)),
	\end{equation}
	for some $Z_1 = Z_1(\zeta)>0$. From now on we take $n\geq n_0$. Let $p_M(X) = X^t - c_{t-1}X^{t-1} - \dots - c_0 \in \Z[X]$ be the minimal polynomial of the matrix $M$, that is, the monic polynomial $q$ of lowest degree in $\mathbb{Q}[X]$ such that $q(M) = 0$. It is a consequence of the Gauss lemma that in fact, being $M$ an integer matrix, $p_M$ has integer coefficients. Since all roots of $p_M$ are simple, its companion matrix $\mathfrak{C}_{p_M}$ is diagonalizable and has simple eigenvalues, where
	{\small
	\[
	\mathfrak{C}_{p_M} = \begin{pmatrix}
	0 & 1 & 0 & \dots & 0 \\
	0 & 0 & 1 & \dots & 0 \\
	\vdots & \vdots & \vdots & \ddots & \vdots\\
	0 & 0 & 0 & \dots & 1 \\
	c_{0}& c_{1}& c_{2}& \dots & c_{t - 1}
	\end{pmatrix}.
	\]}
	
	We have for any $n\geq 0$, $M^n p_M(M) \widetilde{\om} = 0$. This implies
	\begin{align*}
		&M^{n+t}\widetilde{\om} = c_{t-1}M^{n+t-1}\widetilde{\om} + \dots + c_{0}M^{n}\widetilde{\om}\\
		&\implies {\bK}_{n+t}^V - \sum_{i=0}^{t-1} c_i {\bK}_{n+i}^V = - \bm{\epsilon}^V_{n+t} + \sum_{i=0}^{t-1} c_i \bm{\epsilon}^V_{n+i}.
	\end{align*}
	Since all $c_i\in \Z$, the left-hand side of the last equation belongs to $\Z^N$. Then, $\norm{\bm{\epsilon}^V_{n}}_\infty, \dots,$ $\norm{\bm{\epsilon}^V_{n+t}}_\infty < \delta_0 := 1/\left( 1+\sum_{i=0}^{t-1} \abs{c_i} \right)$ implies 
	\small{
	\begin{equation}\label{recurrence}
	\underbrace{ \begin{pmatrix}
	(\bm{\epsilon}^V_{n+1})^{\sf{T}}\\
	\vdots\\
	(\bm{\epsilon}^V_{n+t})^{\sf{T}}
	\end{pmatrix} }_{=:\bm{\mathcal{E}}_{n+1} } = \mathfrak{C}_{p_M} \underbrace{ \begin{pmatrix}
	(\bm{\epsilon}^V_{n})^{\sf{T}}\\
	\vdots\\
	(\bm{\epsilon}^V_{n+t-1})^{\sf{T}}
	\end{pmatrix} }_{=:\bm{\mathcal{E}}_{n}  }.
	\end{equation}}Let us call $\bm{\mathcal{E}}_{n}^{{\bv}_i}$ to the $i$th column of $\bm{\mathcal{E}}_{n} $. In a similar fashion, consider the column vectors ${\bK}_n^{{\bv}_i}$ of the matrix whose rows are given by the vectors $({\bK}_{n}^V)^{\sf{T}},\dots,({\bK}_{n+t-1}^V)^{\sf{T}}$, that is,
	\small{
	\[
	\left( {\bK}_n^{{\bv}_1},\dots, {\bK}_n^{{\bv}_N} \right) = \begin{pmatrix}
	({\bK}_{n}^V)^{\sf{T}}\\
	\vdots\\
	({\bK}_{n+t-1}^V)^{\sf{T}}
	\end{pmatrix}.
	\]}We remark that for all $n\geq0$ there exists ${\bv}_i$ such that $\bm{\mathcal{E}}_{n}^{{\bv}_i} \neq {\bzero}$, since otherwise equation (\ref{decomposition}) would imply $\bm{\mathcal{E}}_{n+m}^{{\bv}_i} = {\bzero}$ for all $m\geq0$. This implies $\om$ is a non-trivial eigenvalue of $(\mathfrak{X_{\zeta}},\phi)$ by Theorem \ref{thm:eigenvalues}, since the same would apply for any return vector, and in consequence, a contradiction to Theorem \ref{wmcondition}.
		
	Let $\{\bbe_{k}\}^{t}_{k=1}$, $\{\bbe^{*}_{k}\}^{t}_{k=1}$ be the eigenbasis of $\mathfrak{C}_{p_M}$ and its dual basis respectively given by
	\small{
	\[
	\bbe_k = \begin{pmatrix}
	1\\
	\lambda_k \\
	\vdots \\
	\lambda_k^{t-2} \\
	\lambda_k^{t-1}
	\end{pmatrix}
	, \quad \bbe^{*}_k = \begin{pmatrix}
	c_0 \lambda_k^{t-2} \\
	c_1 \lambda_k^{t-2} + c_0 \lambda_k^{t-3} \\
	\vdots \\
	c_{t-2}\lambda_k^{t-2} + \dots + c_1\lambda_k + c_0\\
	\lambda_k^{t-1}
	\end{pmatrix}.
	\]}Let $\Lambda=\Lambda_1\cup\dots\cup \Lambda_p$ be the spectrum of $L$. Write $\C^t = U\bigoplus W$, where $U$ and $W$ are direct sums of eigenspaces of $\mathfrak{C}_{p_M}$, and the subspace $U$ corresponds to the direct sum of eigenspaces associated to the eigenvalues of $L$. Since $\Lambda$ is a strongly totally non-Pisot family, there exists an eigenvector $\bbe_{j}\in W$ such that the associated eigenvalue satisfies $\abs{\lambda_j}>1$. We have
	\[
	\bm{\mathcal{E}}_{n}^{{\bv}_i} \cdot \bbe^{*}_{j} = - {\bK}_n^{{\bv}_i} \cdot \bbe^{*}_{j}.
	\]	
	Since ${\bK}_n^{{\bv}_i}$ has integer entries, ${\bK}_n^{{\bv}_i} \cdot \bbe^{*}_{j}$ is equal to $Q^{{\bv}_i}_n(\lambda_j)$, for some integer polynomial $Q^{{\bv}_i}_n$ of degree $t-1$. In particular, $Q^{{\bv}_i}_n(\lambda_j)\neq0$, since the minimal polynomial of $\lambda_j$ divides $p_M$ which is of degree $t$ and $\text{gcd}(t,t-1) = 1$, that is, $Q^{{\bv}_i}_n$ can not divide $p_M$. Since $V$ and $t$ only depend on the substitution $\zeta$, there exists a constant $h = h(\zeta)$ such that
	\[
	H(Q^{{\bv}_i}_n) \leq h\norm{\om}_V\norm{L}^{n},
	\]
	where for a polynomial $P(x) = b_0 + \dots + b_r x^r$, we denote its height by $H(P) = \max_{k=0,\dots,r}\, \abs{b_k}$. Then
	\[
	\bm{\mathcal{E}}_{n}^{{\bv}_i} = \sum_{k=1}^t a_k \bbe_{k} \implies
	\abs{a_j} = \dfrac{ \abs{\bm{\mathcal{E}}_{n}^{{\bv}_i} \cdot \bbe^{*}_{j} }}{ \abs{\bbe_{j} \cdot \bbe^{*}_{j} }} = \dfrac{ \abs{ {\bK}_n^{{\bv}_i} \cdot \bbe^{*}_{j} } }{ \abs{\bbe_{j} \cdot \bbe^{*}_{j} }} = \dfrac{ \abs{Q_n^{{\bv}_i}(\lambda_j) } }{ \abs{\bbe_{j} \cdot \bbe^{*}_{j} }}.
	\]
	
	If $\norm{\bm{\mathcal{E}}_{n}^{{\bv}_i}}_\infty,\dots,\norm{\bm{\mathcal{E}}_{n+m}^{{\bv}_i}}_\infty < \delta_0$, for some $m\geq 1$, then 
	\[
	\bm{\mathcal{E}}_{n+m}^{{\bv}_i}  = \mathfrak{C}_{p_M}^m \bm{\mathcal{E}}_{n}^{{\bv}_i}.
	\]
	
	Let us consider the norm $\norm{\cdot}$ of $\C^t$ corresponding to the maximum of the absolute values of the coefficients in the expansion in the eigenbasis $\{\bbe_{k}\}^{t}_{k=1}$. The inequality from  Lemma 1.51 in \cite{Garsia}, allow us to bound from below the value of $\abs{Q_n^{{\bv}_i}(\lambda_j)}$, yielding
	\[
	\norm{\bm{\mathcal{E}}_{n+m}^{{\bv}_i}} \geq \abs{a_j}\abs{\lambda_j}^{m} \geq c\norm{\om}^{-t}_V\norm{L}^{-nt}\abs{\lambda_j}^m,
	\]
	where $c = c(\zeta)>0$. Since all norms in $\C^t$ are equivalent and $\norm{\bm{\mathcal{E}}_{n+m}^{{\bv}_i}}_\infty \leq 1/2$, the last inequality implies that $m\leq Z_2(n + \log(\norm{\om}_V))$, for some integer constant $Z_2=Z_2(\zeta)>1$ only depending on the substitution. Consider $n_1 = n_1(\om) = \max(n_0,Z_2\log(\norm{\om}_V))$. Set $Z = \max(Z_1,Z_2)$. We have proved that for all $n\geq n_1$
	\begin{equation}
		\max(\norm{\bm{\mathcal{E}}_{n}^{{\bv}_i}}_\infty, \dots, \norm{\bm{\mathcal{E}}_{Z n}^{{\bv}_i}}_\infty) \geq \delta_0.
	\end{equation}
	\textbf{Case 1:} $\log(R)^{\mathfrak{m}\delta^2_0/2}\geq \log (1+\max(\norm{\om}_V^Z, \norm{\om}_V^{-Z}))$. By Corollary \ref{cor:sumbound} we have
	\small{
	\begin{align*}
		&\abs{S^f_R(\bX,\om)} \\
		&= \mathcal{O} \left(\norm{f}_{\text{Lip}} R^d \prod^{\floor{\log_{\norm{L}}(R)}}_{l=0} 1 - \max_{{\bv}\in \mathfrak{GR}(\zeta)} \mathfrak{m} \norm{\om \cdot L^l{\bv}}^2_{\R/\Z} \right)\\
		&= \mathcal{O} \left( \norm{f}_{\text{Lip}} R^d \exp\left( -\mathfrak{m}\sum^{\floor{\log_{\norm{L}}(R)}}_{l=0} \max_{i=1,\dots,N} \norm{\om \cdot L^l{\bv}_i}^2_{\R/\Z} \right) \right)\\
		&= \mathcal{O} \left( \norm{f}_{\text{Lip}} R^d \left( \log_{\norm{L}}(R) \right)^{-\mathfrak{m}\delta_0^2/\log(Z)} \log(1+\max(\norm{\om}^Z_V,\norm{\om}_V^{-Z})^{1/\log(Z)} \right)\\
		&= \mathcal{O} \left( \norm{f}_{\text{Lip}} R^d \left( \log(R) \right)^{-\mathfrak{m}\delta_0^2/\log(Z)} \left( \log(R) \right)^{\mathfrak{m}\delta_0^2/2\log(Z)} \right) \\
		&= \mathcal{O} \left( \norm{f}_{\text{Lip}} R^d \left( \log(R) \right)^{-\mathfrak{m}\delta_0^2/2\log(Z)} \right)\\
		&= \mathcal{O} \left( \norm{f}_{\text{Lip}} R^d \left( \log(R) \right)^{-\mathfrak{m}\delta_0^2/2Z}. \right).
	\end{align*}}
	\textbf{Case 2:} $\norm{\om}_V\geq 1$ and $\log(R)^{\mathfrak{m}\delta^2_0/2}\leq \log (1+\norm{\om}_V^Z)$. Again, by Corollary \ref{cor:sumbound}
	\begin{align*}
	\abs{S^f_R(\bX,\om)} &= \mathcal{O} \left(\norm{f}_{\text{Lip}} (1+\norm{\om}_V)^{-1} R^d \right)\\
	&= \mathcal{O} \left(\norm{f}_{\text{Lip}} R^{d} \log(R)^{-\mathfrak{m}\delta^2_0/2Z} \right).
	\end{align*}
	In both cases, $\abs{S^f_R(\bX,\om)} = \mathcal{O}(\norm{f}_{\text{Lip}} R^d \log(R)^{-\eta})$ for all $R\geq R_0 = R_0(\norm{\om}_V)$, for $\eta := \mathfrak{m}\delta^2_0/2Z > 0$ and $R_0$ defined by 
	\[
	\log(R)^{\eta} = \log (1+\max(\norm{\om}_V^Z, \norm{\om}_V^{-Z})).
	\]
	By Lemma \ref{lemmma: spectralmeasure}, we can deduce that there exist $C,\gamma>0$ depending only on $\zeta$ such that
	\[
	\sigma_f (C^{\om}_r) \leq C\norm{f}^2_{\text{Lip}}\log(1/r)^{-\gamma}.
	\]
	for all $0<r<r_0:=1/2R_0$, finishing the proof.
\end{proof} 
\begin{remark}
	Note that following Case 2 in the proof above, the dependence of $r_0$ on $\norm{\om}$ is only on a lower bound of this quantity. This means that to get uniform bounds in the whole $\R^d$ to prove Theorems \ref{thm:main} and \ref{thm:mainselfaffine}, we only need to obtain estimates of $\sigma_f$ on a neighborhood of $\om = {\bzero}$. This is done in the next two sections separately for the self-similar and the self-affine case.
\end{remark}

\section{Uniform bounds: self-similar case}
In this section we prove Theorem \ref{thm:main}, which gives uniform estimates of the spectral measures in all $\R^d$. This is done by combining the estimates of the spectral measures obtained in Proposition \ref{prop:spectralestimate} with estimates at $\om={\bzero}$ in Corollary \ref{cor:estimateorigin} below, which are easily deduced from deviations of usual Birkhoff integrals for self-similar tilings proved in \cite{bufetov2013limit}. We deduce in addition Theorem \ref{thm:corr} from Theorem \ref{thm:main}. In addition to the hypotheses from Section 3, we also assume that the substitution is self-similar. We recall that we denote by $C^{\om}_r$ the set $[\omega_1-r,\omega_1+r]\times\dots\times[\omega_d-r,\omega_d+r]$, for $\om = (\omega_1,\dots,\omega_d)\in\R^d$.

\subsection{Local estimate at $\om = {\bzero}$}
The estimates at the origin are a direct consequence on the deviation of Birkhoff integrals. Different phenomena occur depending on the two largest eigenvalues, more details are found in \cite{bufetov2013limit}. The next statement covers all cases.
\begin{theorem}[A. Bufetov and B. Solomyak,  \cite{bufetov2013limit}, Corollary 4.5] \label{thm:devselfsimilar}
	Let $f$ be a cylindrical function with zero mean. There exist $C>0$ and $\alpha\in (d-1,d)$ depending only on the substitution such that for $R\geq 2$ and any $\bX\in\mathfrak{X}_\zeta$,
	\begin{equation*}
		\left| \int_{C^{\bzero}_R} f(\phi_{{\bs}}(\bX)) d{\bs}  \right| \leq C \norm{f}_1 R^\alpha.
	\end{equation*}
\end{theorem}
From this result it easy to deduce, by means of Lemma \ref{lemmma: spectralmeasure}, an upper bound for the spectral measure at $\om = {\bzero}$.
\begin{corollary}\label{cor:estimateorigin}
	Let $f$ be a cylindrical function with zero mean. There exist constants $c,\gamma > 0$, only depending on the substitution, such that for all $r < 1/2$
	\begin{equation}
		\sigma_f(C^{\bzero}_r) \leq c\norm{f}_{1}^2r^\gamma .
	\end{equation}
\end{corollary}
The next theorem is a natural generalization to higher dimensions of Theorem 5.1 in \cite{bufetov2014modulus}.
\begin{theorem}\label{thm:main}
	Let $\zeta$ be a self-similar and strongly totally non-Pisot substitution. Let $f\in L^2(\mathfrak{X}_\zeta,\mu)$ be a Lipschitz and zero mean cylindrical function. There exist constants $r_0,c,\gamma>0$ only depending on the substitution, such that for any $\om\in\R^d$
	\[
	\sigma_f (C^{\om}_r) \leq c \norm{f}^2_{\textup{Lip}}\log(1/r)^{-\gamma},\quad  \forall r< r_0.
	\]
\end{theorem}

\begin{proof}
	By Proposition \ref{prop:spectralestimate} and Corollary \ref{cor:estimateorigin}, the only case left to study is $\om\neq0$, $\norm{\om}^Z_V < 1$ and $\log(R)^{\mathfrak{m}\delta^2_0/2}\leq \log (1+\norm{\om}_V^{-Z})$. This last relation implies
	\[
	\norm{\om}_V < \log((2r)^{-1})^{-\eta},
	\]
	where $r = 1/2R$ and $\eta = \mathfrak{m}\delta^2_0/2Z$. Since $\sigma_f$ is a positive measure, we deduce from Corollary \ref{cor:estimateorigin} that  
	\begin{align*}
		\sigma_f(C^{\om}_r) \leq \sigma_f(C^{{\bzero}}_{\norm{\om}_{V}+r})
		\leq C\norm{f}_{1}^2(\norm{\om}_{V}+r)^\gamma
		&\leq C\norm{f}_{1}^2 (\log((2r)^{-1})^{-\eta}+r)^\gamma\\
		&\leq \widetilde{C} \norm{f}_{1}^2 \log(1/r)^{-\tilde{\gamma}},
	\end{align*}
	for some $\widetilde{C}, \tilde{\gamma} > 0$ only depending on the substitution. Finally, we note that for $f$ of the form \eqref{cylindrical}, we have
	\small{
	\[
	\norm{f}_1 = \sum_{i=1}^m \mu(\mathfrak{X}_{T_i}) \norm{\psi_i}_1 = \mathcal{O}(\norm{f}_{\infty}),
	\]}since all $\psi_i$ have compact support.
\end{proof}

\begin{theorem}\label{thm:corr}
	Let $\zeta$ be a self-similar and strongly totally non-Pisot substitution. Let $f,g\in L^2(\mathfrak{X}_\zeta,\mu)$ and suppose $f$ is Lipschitz and zero mean cylindrical function. There exist constants $C,\alpha>0$ only depending on the substitution, such that for all $R\geq2$
	\[
	\dfrac{1}{R^d} \int_{C^{{\bzero}}_R} \left| \langle f \circ \phi_{{\bs}}, g \rangle_{L^2(\mathfrak{X}_\zeta,\mu)} \right|^2 d{\bs} \leq C\norm{f}^2_{\textup{Lip}}\norm{g}^2_2\log(R)^{-\alpha}.
	\]
\end{theorem}

\begin{proof} 
	This follows from Theorem \ref{thm:main} by standard methods. Below, $\langle\cdot,\cdot\rangle$ denotes the standard inner product in $L^2(\mathfrak{X}_\zeta,\mu)$.
	\small{
	\begin{align*}
		\int_{C^{{\bzero}}_R} \left| \langle f \circ \phi_{{\bs}}, g \rangle \right|^2 d{\bs} &= \int_{C^{{\bzero}}_R} \langle f \circ \phi_{{\bs}}, g \rangle \int_{\R^d} e[\om\cdot{\bs}] d\sigma_{f,g}(\om) d{\bs}\\
		&= \int_{\R^d} \left\langle S^f_R(\cdot,\om) ,g \right\rangle d\sigma_{f,g}(\om)\\
		&= \underbrace{\int_{\Omega} \left\langle S^f_R(\cdot,\om) ,g \right\rangle d\sigma_{f,g}(\om)}_{=: I_1} + 
		\underbrace{\int_{\R^d\setminus \Omega} \left\langle S^f_R(\cdot,\om) ,g \right\rangle d\sigma_{f,g}(\om)}_{=:I_2},
	\end{align*}}where $\Omega = \left\{\om\in \R^d\,\middle|\, \norm{\om}_V < 1,  \log(R)^{\mathfrak{m}\delta^2_0/2}\leq\log (1+\norm{\om}^{-Z}_V)\right\}$, and all the relevant constants are defined in the proof of Proposition \ref{prop:spectralestimate}. Following Case 1 and Case 2 in this proposition, for $\om\notin\Omega$
	\[
	\norm{S^f_R(\cdot,\om)}_2 \leq \norm{S^f_R(\cdot,\om)}_\infty =  \mathcal{O} \left( \norm{f}_{\text{Lip}} R^d \left( \log(R) \right)^{-\mathfrak{m}\delta_0^2/2Z} \right),
	\]
	and in consequence, since $|\sigma_{f,g}(\R^d)|\leq \norm{f}_2\norm{g}_2$, this yields
	\[
	|I_2| = \mathcal{O}\left( \norm{f}_{\text{Lip}}\norm{f}_2 \norm{g}^2_2 R^d \left( \log(R) \right)^{-\mathfrak{m}\delta_0^2/2Z} \right).
	\]
	To deal with $I_1$, we follow the proof of Theorem \ref{thm:main}: for $\om\in\Omega$ we have $\norm{\om}_V < \log(R)^{-\eta}$, yielding
	\small{
	\begin{align*}
		|I_1| = \mathcal{O}\left( R^d\norm{f}_\infty\norm{g}_2 |\sigma_{f,g}(\Omega)| \right)
		&= \mathcal{O}\left( R^d\norm{f}_\infty\norm{g}_2 \sqrt{\sigma_f(\Omega)}\sqrt{\sigma_g(\Omega)}  \right)\\
		&= \mathcal{O}\left( R^d\norm{f}_\infty\norm{g}^2_2 \sqrt{\sigma_f \left(C^{\bzero}_{\log(R)^{-\eta}} \right)}  \right)\\
		&= \mathcal{O}\left( R^d\norm{f}^2_\infty\norm{g}^2_2 \log(R)^{-\tilde{\gamma}} \right).
	\end{align*}}
	The result follows immediately.
\end{proof}

\section{Self-affine case}
In this section we prove similar results to Theorems \ref{thm:main} and \ref{thm:corr}, but in the context of self-affine substitutions. Although this case follows the same steps of the self-similar one, the main difference is the local estimate at $\om = {\bzero}$. A result of S. Schmieding and R. Treviño in \cite{schmieding2018self} gives an upper bound for the deviations of ergodic averages, like Theorem \ref{thm:devselfsimilar}, but for the self-affine case. However, to apply this result we need to make some extra assumptions on the substitution and on the functions. In addition, the natural domain of the ergodic averages is not the cube $C^{\bzero}_R$ as in Theorem \ref{thm:devselfsimilar}, but a deformation of it given by the exponential map orbit that passes through the expansion $L$ of the substitution.

Let us summarize the new hypotheses in the next lines. First, we suppose the expansion associated with the substitution is diagonalizable over $\C$. Second, we assume the supports of the tiles are given by polyhedra, and that they intersect on full edges (this enforces finite local complexity for obvious reasons). The latter assumption implies the pattern equivariant cohomology of the tiling space is finite dimensional, and the associated cohomology action is simply a linear transformation between finite dimensional spaces (see Corollary 1 in \cite{schmieding2018self}), and this plays a fundamental role in the proof of the Theorem \ref{thm:devselfaffine} below. For an introduction on the cohomology of tiling spaces, the reader can consult \cite{sadun2008topology}.

Let us recall some of the notions needed to state the main result from \cite{schmieding2018self}. Let $C^{{\bzero}}_1= [-1,1]^d$. Let $L$ be the linear expansion of the substitution $\zeta$, which we may suppose satisfies $\det(L) > 0$ by passing to a power. Let $\mathfrak{a}$ be the matrix that satisfies $\exp(\mathfrak{a}) = L$, and set $g_R = \exp(R\mathfrak{a})$. Finally, let
\[
B_R = g_{\sigma\log(R)} C^{{\bzero}}_1,
\]
where the constant $\sigma = d/\log(\det(L))$ is chosen such that $\text{Vol}(B_R) = R^d \text{Vol}(C^{{\bzero}}_1)$. In particular, $\det (g_{\sigma\log(R)}) = R^d$.
\begin{theorem}[S. Schimieding and R. Treviño, \cite{schmieding2018self}, Theorem 1]\label{thm:devselfaffine}
	Let $f$ be a smooth cylindrical function with zero mean. There exist $C>0$ and $\alpha\in (d-1,d)$ depending only on the substitution such that for $R\geq 2$ and any $\bX\in\mathfrak{X}_\zeta$,
	\begin{equation*}
	\left| \int_{B_R} f(\phi_{{\bs}}(\bX)) d{\bs}  \right| \leq C \norm{f}_\infty R^\alpha.
	\end{equation*}
	Moreover, $\alpha$ only depends on the induced action (of the substitution) on the cohomology space $H^d(\mathfrak{X}_\zeta,\C)$.
\end{theorem}
As in the self-similar case, this result implies a local estimate of the spectral measure in a neighborhood of zero. For completeness we prove this in the next lemma, which is just a slight variation of Lemma \ref{lemmma: spectralmeasure} in the particular case we are interested. Define
\[
\widetilde{S}_R^f(\bX,\om) := \int_{B_R} e[ \om \cdot {\bs} ]f(\phi_{{\bs}}(\bX)) d{\bs}, \quad \widetilde{G}_R^f(\om) = \dfrac{1}{ \text{Vol}(B_R) }\norm{\widetilde{S}_R^f(\cdot,\om)}^2_{L^2(\mathfrak{X}_\zeta,\mu)}.
\]
\begin{lemma}
	If $\widetilde{G}_R^f({\bzero}) \leq CR^d R^\gamma$ for all $R\geq R_0$, for some $C,R_0,\gamma>0$, then
	\[
	\sigma_f \left( M_{r}^{-1} C^{{\bzero}}_r\right) \leq cr^{\gamma},
	\]
	for all $r\leq r_0=1/2R_0$ and some $c>0$, where $M_r = 2r\,g^{\sf T}_{-\sigma\log(2r)}$.
\end{lemma}
\begin{proof}
	We have
	\begin{align*}
	\widetilde{G}_R^f({\bzero}) &= \int_{\R^d} \dfrac{1}{ \text{Vol}(B_R) } \int_{B_R}\int_{B_R} e\left[ \boldsymbol{\lambda}\cdot({\bs}-\boldsymbol{t}) \right] d{\bs}\,d\boldsymbol{t}\,d\sigma_f(\boldsymbol{\lambda}).
	\end{align*}
	Note that by a change of variables we have the next equality
	\[
	\int_{B_R} e \left[ \bx\cdot\boldsymbol{y} \right] d\boldsymbol{y} = \int_{C^{\bzero}_1} e\left[  g^{\sf T}_{\sigma\log(R)}\bx\cdot\boldsymbol{y} \right] \underbrace{ \det(g^{\sf T}_{\sigma\log(R)}) }_{= R^d} d\boldsymbol{y}.
	\]
	Continuing with the former calculation,
	\begin{align*}
	\widetilde{G}_R^f({\bzero}) &= \int_{\R^d} \dfrac{R^d}{ \text{Vol}(C^{\bzero}_1) }\int_{C^{\bzero}_1}\int_{C^{\bzero}_1} e \left[ g^{\sf T}_{\sigma\log(R)} \boldsymbol{\lambda}\cdot({\bs}-\boldsymbol{t})  \right] d{\bs}\,d\boldsymbol{t}\,d\sigma_f(\boldsymbol{\lambda}) \\
	&= \int_{\R^d} R^d \mathfrak{F}^d_1(g^{\sf T}_{\sigma\log(R)}\boldsymbol{\lambda}) \,d\sigma_f(\boldsymbol{\lambda}).
	\end{align*}
 	The Fej\'er kernel satisfies the next property: for any constant $\tau>0$ and all $R>0$
	\[
	\mathfrak{F}^d_R(\tau\bx) = \dfrac{1}{\tau^d}\mathfrak{F}^d_{\tau R}(\bx).
	\]
	One last change of variable yields
	\small{
	\begin{align*}
	\widetilde{G}_R^f({\bzero}) = \int_{\R^d} \mathfrak{F}^d_R\left( \underbrace{ \dfrac{1}{R}\,g^{\sf T}_{\sigma\log(R)} }_{= M_r} (\boldsymbol{\lambda}) \right) \,d\sigma_f(\boldsymbol{\lambda})
	&= \int_{\R^d} \mathfrak{F}^d_R (M_r\boldsymbol{\lambda})\,d\sigma_f(\boldsymbol{\lambda})\\
	&= \int_{\R^d} \mathfrak{F}^d_R (\boldsymbol{\lambda})\,d(\sigma_f \circ M_r^{-1})(\boldsymbol{\lambda}).
	\end{align*}}
	By the same arguments given in Lemma \ref{lemmma: spectralmeasure}, we have for $r=1/2R$
	\small{
	\begin{align*}
	\sigma_f\left( M_r^{-1} C^{{\bzero}}_r\right) = \int_{M_r^{-1} C^{{\bzero}}_r} 1 \, d\sigma_f(\boldsymbol{\lambda})
	&= \int_{C^{{\bzero}}_r} 1 \,d(\sigma_f \circ M_r^{-1})(\boldsymbol{\lambda}) \\
	&\leq \dfrac{\pi^{2d}}{(4R)^d} \int_{C^{{\bzero}}_r} \mathfrak{F}^d_R(\boldsymbol{\lambda}) \,d(\sigma_f \circ M_r^{-1})(\boldsymbol{\lambda}) \\
	&\leq \dfrac{\pi^{2d}}{(4R)^d} \widetilde{G}_R^f({\bzero}) \\
	&\leq c r^{\gamma}.
	\end{align*}}
\end{proof}
\begin{corollary}\label{cor:estimateorigin2}
	Let $f$ be a cylindrical function with zero mean. There exist constants $c,\gamma > 0$, only depending on the substitution, such that for all $r < 1/2$
	\begin{equation}
	\sigma_f(M_r^{-1}C^{{\bzero}}_r) \leq c\norm{f}_{\infty}^2r^\gamma.
	\end{equation}
\end{corollary}
Now we can prove the uniform estimate for the spectral measures in the self-affine case.
\begin{theorem}\label{thm:mainselfaffine}
	Let $\zeta$ be a self-affine and strongly totally non-Pisot substitution, whose expansion is diagonalizable over $\C$. Let $f\in L^2(\mathfrak{X}_\zeta,\mu)$ be a smooth and zero mean function. There exist constants $r_0,c,\gamma>0$ only depending on the substitution, such that for any $\om\in\R^d$
	\[
	\sigma_f (C^{\om}_r) \leq c \norm{f}^2_{\infty}\log(1/r)^{-\gamma},\quad  \forall r< r_0.
	\]
\end{theorem}

\begin{proof}
	We start by remarking that the previous corollary actually give us an estimate for the non distorted cube $C^{{\bzero}}_r$. Note that for any invertible linear transformation $T$ we have
	\[
	\min_{\norm{\bx}_\infty=r} \norm{T^{-1}\bx}_\infty \geq \norm{T}^{-1} r.
	\]
	In particular, we have $M_r^{-1}C^{{\bzero}}_r \supseteq C^{{\bzero}}_{\norm{M_r}^{-1}r} $. Note also that
	\[
	\norm{M_r} = 2r\norm{g^{\sf T}_{\sigma\log(1/2r)}} = 2r \norm{\exp (\sigma\log(1/2r) \mathfrak{a})}\leq 2r e^{\sigma\log(1/2r)\norm{\mathfrak{a}}} = (2r)^{1-\sigma\norm{\mathfrak{a}}}.
	\]
	Notice that by the definition of $\sigma$ we have $\sigma\norm{\mathfrak{a}} \geq 1$, since $\det(L) \leq \norm{L}^d = e^{d\norm{\mathfrak{a}}}$. This yields
	\[
	r\norm{M_r}^{-1} \geq r^{\sigma\norm{\mathfrak{a}}}.
	\]
	Finally, this implies
	\[
	M_r^{-1}C^{{\bzero}}_r \supseteq C^{{\bzero}}_{r^{\sigma\norm{\mathfrak{a}}}}.
	\]
	Following the same steps of the proof of Theorem \ref{thm:main} allow us to conclude.
\end{proof}

Once again, we may derive weak mixing rates from the previous result:

\begin{theorem}\label{thm:corr2}
	Let $\zeta$ be a self-affine and strongly totally non-Pisot substitution, whose expansion is diagonalizable over $\C$. Let $f,g\in L^2(\mathfrak{X}_\zeta,\mu)$ and suppose $f$ is smooth and zero mean cylindrial function. There exist constants $C,\alpha>0$ only depending on the substitution, such that
	\[
	\dfrac{1}{R^d} \int_{C^{{\bzero}}_R} \left| \langle f \circ \phi_{{\bs}}, g \rangle_{L^2(\mathfrak{X}_\zeta,\mu)} \right|^2 d{\bs} \leq C\norm{f}^2_{\infty}\norm{g}^2_2\log(R)^{-\alpha}.
	\]
\end{theorem}

\begin{proof}
	The proof is analogous to the one of Theorem \ref{thm:corr}.
\end{proof}


\end{document}